\documentclass[11 pt]{article}

\usepackage{latexsym, amsmath, amssymb, longtable, booktabs,amscd,microtype,booktabs,amsthm,fancyhdr,amsfonts}

\usepackage{etoolbox}
\usepackage[affil-it]{authblk}
\usepackage[square]{natbib}
\usepackage{enumerate}
\usepackage{tikz}
\usetikzlibrary{matrix,arrows,decorations.pathmorphing}
\usepackage{hyperref}
\bibpunct{[}{]}{,}{n}{}{;} 

\numberwithin{equation}{section}
\newtheorem{thm}{Theorem}[subsection]

\newtheorem{cor}[thm]{Corollary}

\newtheorem{prop}[thm]{Proposition}
\newtheorem{lemma}[thm]{Lemma}
\theoremstyle{definition}

\theoremstyle{remark}
\newtheorem{rmk}[thm]{\textbf{Remark}}

\theoremstyle{definition}
\newtheorem{dfn}{Definition}[subsection]

\theoremstyle{remark}

\theoremstyle{remark}

\makeatletter
\def\imod#1{\allowbreak\mkern10mu({\operator@font mod}\,\,#1)}

\makeatother

\title{\textbf{A Brief Introduction To Splitting Of Primes Over Number Fields}}
\author{SUBHAM DE}
\affil{Department of Mathematics}
\affil{Indian Institute of Technology, Delhi, India.}
\affil{Email: \textbf{mas227132@iitd.ac.in}}

\date{}

\begin{document}
	
	\maketitle
	\thispagestyle{empty}

	\begin{abstract}
		The study of \textit{Dedekind Zeta Functions} over a number field extension uses different aspects of both \textit{Algebraic} and \textit{Analytic Number Theory}. In this paper, we shall learn about the structure and different analytic aspects of such functions, namely the domain of its convregence and analyticity at different points of $\mathbb{C}$ when the function is defined over any finite field extension $K$ over $\mathbb{Q}$ . Moreover, given any two Number Fields $L$ and $K$ over $\mathbb{Q}$ with $L$ being Normal over $K$, our intention is to classify and study the primes in $K$ which split completely in $L$. Also, we shall explore some special cases related to this result.
		\texttt{\textbf{2020 MSC: }}\texttt{Primary  11-02, 11M36, 11R42, 11R45, 11S40 }.\\
		\hspace*{53pt}\texttt{Secondary 11R44, 11S15, 32A60, 32A10 }.
	\end{abstract}
	\pagenumbering{arabic}

	\tableofcontents

	\pagestyle{fancy}
	
	\fancyhead[LO,RE]{\markright}
	\lfoot[]{Subham De}
	\rfoot[]{IIT Delhi, India}

	\section{Motivation to study Dedekind Zeta Functions}
	
		\textit{Dedekind Zeta Functions }over an algebraic number field $K$ over $\mathbb{Q}$, usually denoted by $\zeta_{K}(s)$, where, $s\in \mathbb{C}$, is often considered as the generalization of the \textit{Riemann Zeta Functions} over $\mathbb{Q}$ (It is to be noted that, in the case for the \textit{Riemann zeta functions}, we take the field $K$ mentioned above to be $\mathbb{Q}$ itself.)\par
	\textit{Dedekind zeta functions} can be expressed as a \textbf{\textit{Dirichlet Series}}, as well as it has an \textbf{\textit{Euler Product Expansion}}, which we shall discuss thoroughly in the next section.\par
	Worth mentioning that, Dedekind zeta functions satisfy a \textit{functional equation}. Moreover, they can be analytically extended to a \textit{meromorphic function} on the complex plane $\mathbb{C}$ with only a simple pole at $s=1$. Indeed, it helps us conclude about various properties of the number field $K$ over $\mathbb{Q}$.\par
	Another aspect of Dedekind zeta functions is the \textbf{\textit{Riemann Hypothesis}}, which is till date considered as one of the very few unsolved problems in the field of mathematics, the statement of the conjecture is as follows:
	\begin{center}
		\textit{If $\zeta_{K}(s)=0$ and, $0<Re(s)<1$, then, $Re(s)=\frac{1}{2}$}.
	\end{center}

	\section{A Brief Overview of Dirichlet Series}

	\subsection{Introduction}
	Suppose we consider the \textit{Dirichlet's Series} of complex numbers defined by,
	\begin{center}
		$\textit{D(s)}=\sum\limits_{n=1}^{\infty}\frac{a_{n}}{n^{s}}$  ,
	\end{center}
	where, $a_{n}$'s are fixed complex numbers  $\forall$ $n$, and the co-efficients of different terms in the series. \\
	$s=x+iy$;  $x,y\in\mathbb{R}$     is a complex variable.\\
	First we observe the analyticity of the \textit{Dirichlet Series} on different points of $\mathbb{C}$.
	\subsection{Analytic Properties}
	
	\begin{dfn}\label{def1}
		(Holomorphic Functions) A \textit{holomorphic function }is a complex-valued function of one or more complex variables that is complex differentiable in a neighborhood of every point in its domain.\\
		More generally, the term \textit{holomorphic }is used in the same sense as an \textit{analytic function}.
	\end{dfn}
	Let us mention the following lemma regarding analyticity of \textit{Dirichlet Functions} on $\mathbb{C}$ :
	\begin{lemma}\label{lemma1}
		Suppose for the above \textit{Dirichlet Series},  the partial sum of co-efficients, $\sum\limits_{n\leq t}a_{n}$ = $O(t^{r})$ for some real number $r>0$. Then the \textit{Dirichlet Series} converges for every $s\in \mathbb{C}$ with $Re(s)>r$ , and is analytic, as well as holomorphic function on that half-plane.
	\end{lemma}
	\begin{proof}
		Here, for the proof of lemma \eqref{lemma1}, we apply \textit{Weierstrass Theorem} for analyticity of power series.
	\begin{thm}\label{thm1}
		(Weierstrass Theorem) Suppose that, for every $n\in \mathbb{N}$, the function $f_{n}(z)$ is analytic in the region $\Omega_{n}$, also the sequence ${f_{n}(z)}$ converges to a limit function $f(z)$ in a region $\Omega$, and converges uniformly on every compact subset of $\Omega$. Then, $f(z)$ is analytic in $\Omega$, moreover, $f_{n}'(z)\longrightarrow f'(z)$ uniformly on every compact subset of $\Omega$.
	\end{thm}
		As evident from the statement of Theorem \eqref{thm1}, it only suffices to establish the absolute convergence of the series on every compact subset of the half-plane.\\
		Defining the partial summation of the co-efficients as,
		$A_{k}=\sum\limits_{n=1}^{k}a_{n}$, we obtain, for any $m,M\in \mathbb{N}$  with, $1\leq m<M$,
		\begin{eqnarray}\label{1}
			\sum\limits_{n=m}^{N}\frac{a_{n}}{n^{s}}=\sum\limits_{n=m}^{M}\frac{A_{n}}{n^{s}}-\sum\limits_{n=m}^{M}\frac{A_{n}}{n^{s}}\space  =  \space\frac{A_{M}}{M^{s}}-\frac{A_{m-1}}{m^{s}}+\sum\limits_{n=m}^{M-1}A_{n} ( \frac{1}{n^{s}}-\frac{1}{(n+1)^{s}} ).			
		\end{eqnarray}
		Since, by statement, the partial sum is $O(t^{r})$, therefore, $\exists$ a positive $B\in \mathbb{R}$ such that,  $|A_{n}|\leq Bn^{r}$,  $\forall$ $n\in \mathbb{N}$. Applying this condition in equation \eqref{1}, we get,
		\begin{eqnarray}\label{2}
			|\sum\limits_{n=m}^{M}\frac{a_{n}}{n^{s}}|\space\leq\space B\left(\frac{M^{r}}{|M^{s}|}+\frac{(m-1)^{r}}{|m^{s}|}+\sum\limits_{n=m}^{M-1}n^{r}|\frac{1}{n^{s}}-\frac{1}{(n+1)^{s}}|\right)			
		\end{eqnarray} 
		Therefore, 
		\begin{center}
			$|\frac{1}{n^{s}}-\frac{1}{(n+1)^{s}}|\leq |s|\int\limits_{n}^{n+1}\frac{dt}{|t^{s+1}}|=|s|\int\limits_{n}^{n+1} \frac{dt}{t^{x+1}}\leq \frac{|s|}{n^{x+1}}$
		\end{center} 
		( Since, assuming $s=x+iy$, where, $x=Re(s)\in \mathbb{R}$ and, $y\in \mathbb{R}$, it can be written that, $ \frac{1}{n^{s}} - \frac{1}{(n+1)^{s}}=s\int\limits_{n}^{n+1}\frac{dt}{t^{s+1}}$, and hence, the above result follows ).\\
		Which yields, from equation \eqref{2},
		\begin{eqnarray}\label{3}
			|\sum\limits_{n=m}^{M}\frac{a_{n}}{n^{s}}|\space \leq \space B\left(M^{r-x}+m^{r-x}+|s|\sum\limits_{n=m}^{M-1}n^{r-x-1}\right).	
		\end{eqnarray}
		Since, we have, 
		\begin{center}
			$\sum\limits_{n=m}^{M-1}n^{r-x-1}\leq \int\limits_{m-1}^{\infty}t^{r-x-1}dt$
		\end{center}
		 ( This improper integral is valid, by Comparison Test )
		\begin{center}
			$=\lim\limits_{v\rightarrow \infty}\int\limits_{m-1}^{v}t^{r-x-1}dt=\frac{(m-1)^{r-x}}{x-r}$ ,  for any $m>1$,
		\end{center}
	 	 Hence, having put $m\rightarrow \infty$ and $M\rightarrow \infty$, from \eqref{3}, it is evident that,
		\begin{center}
			$|\sum\limits_{n=m}^{M}\frac{a_{n}}{n^{s}}|  \longrightarrow 0$  
		\end{center}
		for any fixed $s=x+iy$ in the half-plane, $x>r$.\\
		This implies that, the partial sum of the \textit{Dirichlet Series }converges, consequently, we conclude that, the series itself converges $\forall$ $s=x+iy$ with $x>r$.\\
		Moreover, since each such compact set in the half-plane
		\begin{center}
			$\{s=x+iy\in \mathbb{C}\mbox{ }|\mbox{ }x,y\in \mathbb{R} \mbox{ \& } Re(s)=x>r\}$
		\end{center}
		is bounded ( we can have, $|s|\leq B_{1}$ for some positive real $B_{1}>0$ ) and is away from the line, $Re(s)=x=r$  (Since, $x>r$ $\implies$ $x-r\geq \epsilon$ for some $\epsilon>0$ ). It implies that, the convergence of the \textit{Dirichlet Series }is also uniform on every compact subset of the half-plane, $x>r$, $x=Re(s)$.\\
		Thus, we have an uniform estimate,
		\begin{center}
			$|\sum\limits_{n=m}^{M}\frac{a_{n}}{n^{s}}| \leq B\left(\frac{1}{m^{\epsilon}}+B_{1}\frac{1}{\epsilon(m-1)^{\epsilon}}\right)$
		\end{center}
		and the proof of the lemma \eqref{lemma1} is complete.
		\end{proof}
		\begin{rmk}\label{rmk1}
				As an appropriate example reflecting the significance of Lemma \eqref{lemma1}, we can observe that, the \textit{\underline{Riemann Zeta Function}} defined by,
			\begin{center}
				$\zeta(s)=\sum\limits_{n=1}^{\infty}\frac{1}{n^{s}}$
			\end{center}
			converges. Since the partial sum of the co-efficients of the above series is $=O(n)$, hence, applying Lemma \eqref{lemma1}, it can be verified that, the function is analytic on the half-plane, $x>1$.
		\end{rmk}
		\section{Dedekind Zeta Function}
		\subsection{Formal Definition}
		\begin{dfn}\label{def2}
			(Dedekind Zeta Function) As a more generalized case of \textit{Riemann Zeta Function} for a number field $K$ over $\mathbb{Q}$, we have the following definition of \underline{\textit{Dedekind  Zeta Function}} for the half-plane, $Re(s)=x>1$,
			\begin{center}
				$\zeta_{k}(s)=\sum\limits_{n=1}^{\infty}\frac{j_{n}}{n^{s}}$, 
			\end{center} 
		 where, $s\in \mathbb{C}$, such that, $Re(s)=x>1$.\\
			And we also define for every  $n\in \mathbb{N}$ ,
			\begin{center}
			$j_{n}=$ Number of ideals $\mathcal{I}$ of the Ring of Integers of $K$ over $\mathbb{Q}$ \\
			\hspace*{80pt}(denoted by, $\mathcal{O}_{K}$) with, $\mathcal{N(I)}=n$ ,
			\end{center}
		\end{dfn}
		\subsection{Analyticity of $\zeta_{K}(s)$}
		For determining the domain of $\zeta_{K}(s)$, where the function is analytic, we introduce an important result,
		\begin{thm}\label{thm2}
			Let $K/ \mathbb{Q}$ be a finite extension of number fields with $[K:\mathbb{Q}]=n$ (say), and $\mathcal{O}_{K}$ be the ring of integers of $K$. We define, for each positive $t\geq0$, $i(t)=$ Number of ideals $\mathcal{I}$ of $\mathcal{O}_{K}$ with $\mathcal{N(I)}\leq t$, and also, for each ideal class $C$ (say), define,
			$i_{C}(t)=$ Number of ideals $I$ in $C$ with $\mathcal{N(I)}\leq t$, i.e., 
			\begin{center}
				$i(t)=\sum\limits_{C}i_{C}(t)$ 
			\end{center}
			It is important to note that, the above sum is finite.\\
			Then $\exists$ a number $\c{k}$, depending on $\mathcal{O}_{K}$ but independent of the ideal class $C$, such that,
			\begin{center}
				$i_{C}(t)=\c{k}t+\epsilon_{C}(t)$
			\end{center}
			Where the error term is denoted by $\epsilon_{C}(t)$, which is $O(t^{1-\frac{1}{n}})$. In other terms, the ratio, $\frac{\epsilon_{C}(t)}{t^{1-\frac{1}{n}}}$ is bounded by a finite real number as $t\rightarrow \infty$.
		\end{thm}
			\begin{rmk}\label{rmk2}
		From the above theorem, we can conclude that, the partial sum of co-efficients in the \textit{Dedekind Zeta Function}, $\sum\limits_{n\leq t}j_{n}$ is $O(t)$. Therefore, by lemma \eqref{lemma1}, $\zeta_{K}$ converges and is analytic on the half-plane, $Re(s)>1$.
		\end{rmk}		
		\begin{lemma}\label{lemma2}
			\textit{Dedekind Zeta Function } $\zeta_{K}$ can be extended to a \textit{meromorphic function} on the half-plane,
			\begin{center}
				$Re(s)>1-\frac{1}{n}$, where, $n=[K:\mathbb{Q}]$,
			\end{center} 
		 where the function is analytic everywhere except at the point $s=1$, where it has a \textit{simple pole} of order 1.
			
		\end{lemma}	
		\begin{rmk}\label{rmk3}
			An analogue of the above statement is that, $(s-1)\zeta_{K}(s)$ is analytic on the entire half-plane, $Re(s)>1-\frac{1}{n}$.
		\end{rmk}	
		\begin{proof}
			Considering the series,
			\begin{center}
				$f(s)=1-\frac{1}{2^{s}}+\frac{1}{3^{s}}-\frac{1}{4^{s}}+......$,
			\end{center}
			we observe that, $f(s)$ converges to an analytic funtion on the entire half-plane, $Re(s)>0$ (Applying Lemma \eqref{lemma1}).\\
			Let, $s=x+iy\in \mathbb{C}$.
			Now, we can write,
			\begin{center}
				$f(s)=(1-2^{1-s})\zeta(s)$     \hspace{20pt}       for, $x>1$ (From Definition of $f(s)$).
			\end{center}
			where, $\zeta(s)$ is the \textit{Riemann Zeta Function} over $\mathbb{Q}$.\\
			Therefore, the function $\frac{f(s)}{1-2^{1-s}}$ is an extension of $\zeta$ to a \textit{meromorphic function} on the half-plane, $x>0$. Hence the function above can have poles, if,
			\begin{center}
				$1-2^{1-s}=0$, \hspace{20pt}  i.e., \hspace{20pt}  $s=1$.
			\end{center}
			Now, we can also observe that, $f(s)=0$ at the points,
			\begin{center}
				$s_{k}=1+\frac{2k\pi i}{log 2}$,\hspace{20pt}    $k=\pm 1,\pm 2,......$
			\end{center}
			Again, considering another series,
			\begin{center}
				$g(s)=1+\frac{1}{2^{s}}-\frac{2}{3^{s}}-\frac{1}{4^{s}}+\frac{1}{5^{s}}-\frac{2}{6^{s}}........$
			\end{center}
			similarly, we can deduce that,
			\begin{center}
				$g(s)=(1-3^{1-s})\zeta(s)$      \hspace{20pt}     for, $x>1$ (From Definition of $g(s)$).
			\end{center}
			and consequently, the function,  $\frac{f(s)}{1-3^{1-s}}$ is an extension of $\zeta$ to a meromorphic function on the half-plane, $x>0$ with poles at every such $s$ satisfying, $1-3^{1-s}=0$,  i.e.,
			\begin{center}
				$s'_{k}=1+\frac{2k\pi i}{log 3}$,\hspace{20pt}    $k=\pm 1,\pm 2,......$
			\end{center}
			Clearly, the set of points $s_{k}$ and $s'_{k}$ are disjoint, $k\neq 0$.\\
			Since, $\lim\limits_{s\rightarrow s_{k}^{+}}\zeta(s)$ is finite, therefore, we have, $\zeta(s)=\frac{f(s)}{1-2^{1-s}}$ does not really have poles at the points $s_{k}$, $k\neq0$, and hence the function is analytic everywhere except at $s=1$. \\
			Assuming this as the definition of the extension of $\zeta$ (also denoting the extension as $\zeta$), we can express $\zeta_{K}$ as, 
			\begin{center}
				$\zeta_{K}(s)=\sum\limits_{n=1}^{\infty}\frac{j_{n}-h\c{k}}{n^{s}}+h\c{k}\zeta(s)$,  for,  $Re(s)=x>1$
			\end{center}
			where,
			\begin{center}
			$h=$ Number of ideal classes in $\mathcal{O}_{K}$\\
			$\c{k}=$ The constant term (as mentioned in Theorem \eqref{thm2}).
			\end{center}
			Applying Theorem \eqref{thm2} and Lemma \eqref{lemma1}, we conclude that, $\zeta_{K}(s)$ can be described as a \textit{Dirichlet Series} with co-efficients, $j_{n}-h\c{k}$,   $\forall n$. Moreover, it converges to an analytic function on the half-plane,
			\begin{center}
				$x>1- \frac{1}{n}$, where, $n=[K:\mathbb{Q}]$
			\end{center}
			(It can be deduced that, the partial sum of the co-efficients, $\sum\limits_{n\leq t}(j_{n}-h\c{k})$ is $O(t^{1- \frac{1}{n}}$).)\\\\
			Combining above with the result obtained for the extension of $\zeta$, the proof is complete.
			\end{proof}
			\subsection{Other Definitions of $\zeta_{K}$}
			
			A priori using the notion of absolute convergence of $\zeta_{K}$, we can redefine $\zeta_{K}$ in the following manner :
			\begin{dfn}\label{def3}
				(Dedekind Zeta Function) We define the \textit{Dedekind Zeta function} as,
				\begin{center}
					$\zeta_{K}(s)=\sum\limits_{I\subset \mathcal{O}_{K}}\frac{1}{\mathcal{N(I)}^{s}}$,\hspace{20pt} for, $Re(s)>1$. 
				\end{center}
				where, we define, 
				\begin{center}
					$\mathcal{N(I)}=$ Norm of an ideal $I\subset\mathcal{O}_{K}$  $= | \mathcal{O}_{K}/I|$,
				\end{center}
			
				and, the sum is performed over every non-zero ideal $I$ of $\mathcal{O}_{K}$.
			\end{dfn}
			\begin{rmk}\label{rmk4}
				Also it is to be noted that the order of the summation is unspecified since it is not required.
			\end{rmk}
			We can further provide a third definition of the \textit{Dedekind Zeta Functions} as,
			\begin{dfn}\label{def4}
				(Dedekind Zeta Function) We define $\zeta_{K}$ as the following product,
				\begin{center}
					$\zeta_{K}(s)=\prod\limits_{\mathcal{P}\subset \mathcal{O}_{K}}(1-\frac{1}{\mathcal{N(P)}^{s}})^{-1}$,\hspace{20pt} for, $Re(s)>1$. 
					
				\end{center}
				where,
				\begin{center}
					$\mathcal{P}:=$ Prime ideal of $\mathcal{O}_{K}$.
				\end{center}
			\end{dfn}
			\begin{rmk}\label{rmk5}
				In definition \eqref{def4}, we are taking the product over every prime ideal $\mathcal{P}$ of $\mathcal{O}_{K}$.
			\end{rmk}
			\begin{rmk}\label{rmk6}
				The product defined above yields a finite value, since, for every prime ideal $\mathcal{P}\subset\mathcal{O}_{K}$, we have, $\mathcal{N(P)}>1$, hence each term in the product is finite.
			\end{rmk}
			
			At this juncture, one may be curious to ask,
			\begin{center}
				\textit{\textbf{Are these definitions of $\zeta_{K}$ equivalent to each other ?}}
			\end{center}
			The answer of course, is \textbf{YES}. We state the following proposition in order to establish our claim.
			\begin{prop}\label{prop1}
				Definitions \eqref{def2},\eqref{def3} and \eqref{def4} of $\zeta_{K}$ are equivalent.
			\end{prop}
			\begin{proof}
				First, we shall establish that, definition \eqref{def3} implies definition \eqref{def2} :\\\\
				We choose $n\in \mathbb{N}$ and define the ideal class of $n$ as:
				\begin{center}
					$[n]:=\{\mathcal{I}\subset\mathcal{O}_{K} \mbox{ }|\mbox{ }\mathcal{N(I)}=n\}$
				\end{center}
				Then, we can observe that, by our previous definition of $j_{n}$, 
				\begin{center}
					$j_{n}=|[n]|$.
				\end{center}   
				Now, taking the sum over every such ideal class corresponding to eevery $n\in \mathbb{N}$, we obatin the definition \eqref{def2} of $\zeta_{K}$.\\\\
				Next, we shall prove that, definition \eqref{def4} implies definition \eqref{def3} :\\\\	
				For every prime ideal,  $\mathcal{P}\subset\mathcal{O}_{K}$, we have previously mentioned that,
				\begin{center}
					$\mathcal{N(P)}>1$ $\implies\frac{1}{\mathcal{N(P)}}<1$.
				\end{center}
				Expanding each term, $\left(1-\frac{1}{\mathcal{N(P)}}\right)^{-1}$ as an infinite series in terms of every such $\mathcal{N(P)}$,
				\begin{center}
					$\zeta_{K}(s)=\prod\limits_{\mathcal{P}}\left(1+\frac{1}{\mathcal{N(P)}}+\frac{1}{\mathcal{N(P)}^{2}}+\frac{1}{\mathcal{N(P)}^{3}}+......\right)$
				\end{center}
				Since, the norm map is multiplicative and every non-zero ideal is either a prime, or it can be expressed as a product of finite powers of prime ideals, therefore, expanding the product over all the prime ideals $\mathcal{P}$ of $\mathcal{O}_{K}$, we obtain the definition \eqref{def3} of $\zeta_{K}$.
				\end{proof}
		\vspace{10pt}
		\section{Density of Primes over Number Field Extensions}
		\vspace{10pt}
		Now, as we have already defined \textit{Dedekind Zeta Function} over a number field $K/\mathbb{Q}$, we proceed towards defining and observing different results regarding density of primes which split completely over a number field extension.\par
		
		\subsection{Important Definitions}
		\begin{dfn}\label{def5}
			(Partial Dedekind Zeta Functions) Let $K/\mathbb{Q}$ be a number field and, we denote,
			\begin{center}
				$\mathcal{A}:=$ Any set of primes of $\mathcal{O}_{K}$.
			\end{center}
			Then, the \textit{Dedekind Partial Zeta Functions} over $K$ is defined as:
			\begin{center}
				$\zeta_{K,\mathcal{A}}(s)=\sum\limits_{\mathcal{I}\in [\mathcal{A}]}\frac{1}{\mathcal{N(I)}^{s}}$ = $\prod\limits_{\mathcal{P}\in\mathcal{A}}\left(1-\frac{1}{\mathcal{N(\mathcal{P})}^{s}}\right)^{-1}$
			\end{center}
			where [$\mathcal{A}$] denotes the semigroup of ideals generated by $\mathcal{A}$; in other words, $\mathcal{I}\in \mathcal{A}$ iff, $\forall$ prime ideals, $\mathcal{P}|\mathcal{I}$  $\implies$  $\mathcal{P}\in \mathcal{A}$.
		\end{dfn}
		\begin{dfn}
			(Polar Density) We have, by definition of \textit{Partial Dedekind Zeta Functions},
			\begin{center}
				$\zeta_{K,\mathcal{A}}(s)=\prod\limits_{\mathcal{P}\in\mathcal{A}}\left(1-\frac{1}{\mathcal{N(\mathcal{P})}^{s}}\right)^{-1}$, \hspace{20pt}for $Re(s)>1$.
			\end{center}
			If, for some integral power $n\geq1$, $\zeta_{K,\mathcal{A}}^{n}$ can be extended to a meromorphic function in a neighbourhood of $s=1$, such that the function shall have a pole of order $m$ at $s=1$, then we define the \textit{Polar Density} of $\mathcal{A}$ to be, $\mathcal{D}=\frac{m}{n}$.
		\end{dfn}
		\begin{rmk}\label{rmk7}
			For example, it can be observed that, a finite set has polar density $\mathcal{D}=0$, and a set containing all but a finitely many elements has polar density $\mathcal{D}=1$.
		\end{rmk}
		\begin{rmk}\label{rmk8}
			If two sets of primes , $\mathcal{A}$ and $\mathcal{B}$ of the number field $K$ differ only by the primes $\mathcal{P}$ such that, for each such $\mathcal{P}$, $\mathcal{N}(\mathcal{P})$ is not a prime, then the polar density of the set $\mathcal{A}$ exists $\Leftrightarrow$ the polar density of $\mathcal{B}$ exists.\par 
			Furthermore, if $\mathcal{D_{A}}$ and, $\mathcal{D_{B}}$ be the \textit{polar densities} of the sets $\mathcal{A}$ and $\mathcal{B}$ respectively, then we shall have,
			\begin{center}
				$\mathcal{D_{A}}=\mathcal{D_{B}}$.
			\end{center}
			(Since the corresponding zeta functions differ by a factor which is analytic and non-zero in the neighbourhood of $s=1$).
		\end{rmk}		
			Having all the neccessary tools we need, we can now study about the corresponding densities of sets of primes that split completely over a finite field extension.
			\subsection{The Main Theorem}

		\begin{thm}\label{thm3}
			 Suppose we consider $L$ and $K$ to be two number fields over $\mathbb{Q}$ such that, the extension, $L/K$ is normal. Then, the set of primes in $K$ which splits completely in $L$ has polar density, $\mathcal{D}=$ $\frac{1}{n}$, where, $n=[L:K]$.
		\end{thm}
		\begin{proof}
			Define,
			\begin{center}
				$\mathcal{A}=\{\mathfrak{p}\subset\mathcal{O}_{K}\hspace*{10pt}| \mbox{ } \mathfrak{p} \mbox{ is prime, \& } \mathfrak{p} \mbox{ splits completely in }L\}$,\\
				$\mathcal{B}=\{\wp\subset\mathcal{O}_{L}\hspace*{10pt}| \mbox{ } \mathfrak{p} \mbox{ is prime, \& } \wp|\mathfrak{p} \mbox{ for some }\mathfrak{p}\in\mathcal{A}\}$.
			\end{center}
		
			Clearly, we can say that, the set $\mathcal{A}$ is infinite (Follows from the fact that, there exists infinitely many primes in a number field $K$ which splits completely over the extension $L$). Hence, the density of $\mathcal{A}$ is non-zero.\\
			Let, $\mathfrak{p}\in\mathcal{A}$ such that, $\exists$ primes, $\wp_{1},\wp_{2},........,\wp_{g}\in \mathcal{B}$ (say)  and,
			\begin{center}
				$\mathfrak{p}=\wp_{1}\wp_{2}........\wp_{g}$
			\end{center}
			Since, $\mathfrak{p}\in\mathcal{A}$ $\implies$ $\mathfrak{p}$ splits completely in $L$, hence, we have, for each, $1\leq i\leq g$, \\
			The inertial degree, 
			\begin{center}
				$f_{i}=$ $[{\mathcal{O}_{L}/\wp_{i}}:{\mathcal{O}_{K}/\mathfrak{p}}]$ $=1$, \hspace*{20pt} $\forall$  $1\leq i\leq g$.\\
				$\implies$ $|\mathcal{O}_{L}/\wp_{i}|=|\mathcal{O}_{K}/\mathfrak{p}|$, for every $1\leq i\leq g$.\\
				$\implies$ $\mathcal{N}(\wp_{i})=\mathcal{N}(\mathfrak{p})$, for every  $1\leq i\leq g$.
			\end{center}
			Again, by a well-known result, we know that, if $e_{i}$ be the \textit{ramification degree} and $f_{i}$ be the corresponding \textit{inertial degree} of each $\wp_{i}$ for every $1\leq i\leq g$, then,
			\begin{center}
				$\sum\limits_{i=1}^{g}e_{i}f_{i}=n=[L:K]$.
			\end{center}
			But, Since $\mathfrak{p}$ splits completely in $L$, hence, $e_{i}=f_{i}=1$  for every $1\leq i\leq g$.
			\begin{center}
				$\implies$  $g=n$.	
			\end{center}
			Similarly, we can obtain similar sort of prime decomposition in $L$.\\
			Therefore, we obtain,
			\begin{center}
				$\zeta_{K,\mathcal{A}}^{n}=\zeta_{L,\mathcal{B}}$.
			\end{center}
			Thus, it only suffices to prove that, $\zeta_{L,\mathcal{B}}$ has a pole of order $1$ at $s=1$.\\\\
			Here, we clearly observe that, $\wp\in \mathcal{B}$ $\implies$ $\mathcal{N}(\wp)$ is prime, and this statement holds for all primes $\wp$ except finitely many, which are ramified over $K$. \par  Therefore,
			\begin{center}
				$\zeta_{L}=\zeta_{L,\mathcal{B}}.H(s)$
			\end{center}
			where, $H(s)$ is a finite factor which is an analytic function and non-zero in a neighbourhood of $s=1$. Therefore, our claim is established and, hence, $\zeta_{K,\mathcal{A}}^{n}$ can be extended to a meromorphic function in a neighbourhood of $s=1$, having pole of order $1$ at $s=1$, hence the polar density of $\mathcal{A}$ is ,
			\begin{center}
				$\mathcal{D_{A}}=\mathcal{D}=\frac{1}{n}=\frac{1}{[L:K]}$.
			\end{center}
			And we have proved our desired result.
			\end{proof}
			\subsection{Some Special Cases}
			As a corollary to the above theorem, we state this result with the intention of obtaining the density of primes in a number field $K$ that split completely in an extension field, say $L$, which is not normal over $K$.
			\begin{cor}\label{cor1}
				Let $L/K$ be any finite extension of number fields with $[L:K]=n$ (say) and, we define, 
				\begin{center}
					$M=$ Normal Closure of $L/K$\\
					\hspace{20pt}$=\overline{L_{N}}$  (say)
				\end{center}
				Then, the set of primes in $K$ that splits completely over $M$ has polar density,
				\begin{center}
					$\mathcal{D}=\frac{1}{[M:K]}$  . 
				\end{center}
			\end{cor}
		\begin{proof}
			Clearly, here, we can observe that, $M/K$ is a normal extension. Let us define the set ,\\
			$\mathcal{A}=\{\mathfrak{p}\subset\mathcal{O}_{K}\hspace*{10pt}|\mbox{ } \mathfrak{p}\mbox{ is prime, \& }\mathfrak{p}\mbox{ splits completely in }M \}$.\\
			Then, applying the result proved in Theorem \eqref{thm3}, we can deduce that, the polar density of $\mathcal{A}$ is , $\mathcal{D_{A}}=$ $\frac{1}{[M:K]}$.\\
			Thus, our only objective is to prove that, the primes in $K$ which split completely over $M$ will also split completely over $L$.\\
			$\Leftrightarrow$ A prime will split completely in $L$ iff, it splits completely in $M$.\\
			Here, we prove an important result:
			\begin{lemma}\label{lemma3}
				Let, $L/K$ be a finite extension ( $[L:K]=n(say)
				$ ) of number fields over $\mathbb{Q}$, and $\mathfrak{p}$ be a prime in $K$. If $\mathfrak{p}$ is unramified or splits completely in $L$, then the same holds in the normal closure $M$ of $L/K$.
			\end{lemma}
			\begin{proof}
				We know that, the normal closure is the composite field of subfields $\sigma L$, where, \\ $\sigma\in Gal(M/L)$, i.e.,
				\begin{center}
					$M=\prod\limits_{\sigma\in Gal(M/L)}\sigma L$
				\end{center}
				Let, $\mathfrak{p}\subset\mathcal{O}
				_{K}$ be a prime in $K$ which splits completely over $L$.\\
				Then, $\exists$ primes,  $\mathfrak{P}_{1},\mathfrak{P}_{2},....,\mathfrak{P}_{g}$ in $\mathcal{O}_{L}$ such that,
				\begin{center}
					$\mathfrak{p}\mathcal{O}_{L}={{\mathfrak{P}_{1}}}{{\mathfrak{P}_{2}}}........{{\mathfrak{P}_{g}}}$
				\end{center}
				Under any automorphism map $\sigma\in Gal(M/L)$, we get, from above,
				\begin{center}
					$\mathfrak{p}\mathcal{O}_{\sigma L}={({\sigma\mathfrak{P}_{1}}})({{\sigma\mathfrak{ P}_{2}}})........({{\sigma\mathfrak{P}_{g}}})$
				\end{center}
				(Since $\mathfrak{p}$ is invariant under $\sigma$.)\\
				Since, $\mathfrak{p}$ splits completely over $L$, therefore, we obtain that,  $\mathfrak{p}$ splits completely in $M$.\\
				Conversely, applying similar process, we can conclude that, if a prime splits completely in $M/K$, then it splits completely in $L/K$ also.\\
				Hence our Lemma \eqref{lemma3} is proved.
			\end{proof}
				Applying Lemma \eqref{lemma3}, we have proved the Corollary \eqref{cor1} .
			\end{proof}
			\begin{cor}\label{cor2}
				$K/\mathbb{Q}$ is a number field, and, $f$ is a monic and irreducible polynomial over $\mathcal{O}_{K}$. Define,
				\begin{center}
					$\mathcal{A}=$ $\{\mathfrak{p}\in \mathcal{O}_{K}\mbox{ }|\mbox{ }f\mbox{ splits into linear factors over }\mathcal{O}_{K}/\mathfrak{p}\}$
				\end{center}
				If $L/K$ is the splitting field of $f$, then, the set $\mathcal{A}$ has polar density,
				\begin{center}
					$\mathcal{D_{A}}=\frac{1}{[L:K]}$
				\end{center}
			\end{cor}
			\begin{proof}
				Fixing one of the roots of the polynomial $f$, and denoting it by $\alpha$ (say), we intend to observe the primes $\mathfrak{p}$ $\in K$ which splits in the extension $K[\alpha]/K$.\\
				Applying a well-known theorem, we can say that, for infinitely many primes $\mathfrak{p}\in K$, $\mathfrak{p}$ splits completely in $K[\alpha]$  $\Longleftrightarrow$  $f$ splits into linear factors over $\mathcal{O}_{K}/\mathfrak{p}$.\\
				Since, $L/K[\alpha]/K$ is the normal closure of $K[\alpha]$, hence, applying corollary \eqref{cor1}, we obtain, by definition of the set $\mathcal{A}$,
				\begin{center}
					$\mathcal{D_{A}}=\frac{1}{[L:K]}$
				\end{center}
				And the statement of the corollary \eqref{cor2} is established.
			\end{proof}
			
			\begin{cor}\label{cor3}
				Let, $H<\mathbb{Z}_{m}^{\times} $, then, the set,
				\begin{center}
					$\mathcal{A}=\{\mathfrak{p}\in \mathbb{Z}\mbox{ }| \mbox{ }\mathfrak{p} \mbox{ is a prime, \& \hspace{10pt}}\mathfrak{p}(mod\mbox{ }m)=\overline{\mathfrak{p}}\in H\}$
				\end{center}
				has polar density,
				\begin{center}
					$\mathcal{D_{A}}=\frac{|H|}{\varphi(m)}$.
				\end{center}
				\textbf{Note: }$\overline{\mathfrak{p}}$ denotes the congruence class of $\mathfrak{p}(mod\mbox{ }m)$.
			\end{cor}
			\begin{proof}
				Consider the cyclotomic extension, $\mathbb{Q}(\zeta_{m})/\mathbb{Q}$, where, $\zeta_{m}$ is the primitive $m^{th}$ root of unity, which is a normal extension, as well as separable. Thus it has Galois Group, $Gal(\mathbb{Q}(\zeta_{m})/\mathbb{Q})=\mathbb{Z}_{m}^{\times}$.\\
				Suppose , we define $L$ to be the fixed field of $H$, then we have, for any prime, $\mathfrak{p} \nmid m$, $\mathfrak{p}$ splits completely in $L$ $\Leftrightarrow$ $\overline{\mathfrak{p}}\in H$. Hence, the result follows.
			\end{proof}
			\begin{cor}\label{cor4}
				A normal extension of a number field $K$ can be uniquely characterized by the set, say $\mathcal{A}$ of primes $\mathfrak{p}\in K$, which split completely in it.\\
				An equivalent statement of the above corollary is, $\exists$ a one-one inclusion-reversing correspondence between the normal extension $L/K$ and the set of primes defined above as $\mathcal{A}$.
			\end{cor}
		\begin{proof}
			Suppose, $L/K$ and $L'/K$ be two normal extensions, which corresponds to the same set $\mathcal{A}$ of primes in $K$ that split completely in respectable extension fields. Therefore, we can say that,
			\begin{center}
				$\mathcal{A}=\{\mathfrak{p}\in K\mbox{ }|\mbox{ }\mathfrak{p} \mbox{ is a prime, \& \hspace{10pt}}\mathfrak{p}\mbox{ splits completely in }L\}$\\
				$=\{\mathcal{P}\in K\mbox{ }|\mbox{ }\mathcal{P} \mbox{ is a prime, \& \hspace{10pt} }\mathcal{P}\mbox{ splits completely in }L'\}$.
			\end{center}
			Suppose that, $M=LL'$. Then, by theorem, we can conclude that,
			\begin{center}
				$\mathcal{A}=\{\wp\in K\mbox{ }|\mbox{ }\wp \mbox{ is a prime, \& \hspace{10pt} }\wp\mbox{ splits completely in }M\}$.
			\end{center}
			Consequently, having the same polar density, we have, by theorem \eqref{thm3},
			\begin{center}
				$[M:K]=[L:K]=[L':K]$ \\
				$\Rightarrow$  $L=L'$. 
			\end{center}
			And hence, the corollary \eqref{cor4} holds true.
			\end{proof}
			\begin{rmk}\label{rmk9}
				In this paper, we have only observed about how a set of primes over a number field looks like in any finite field extension, and also applied the analytic concepts of \textit{Dedekind Zeta Functions} over number fields to obtain an expression for the density of such sets. This article also reflects densities of such sets of primes for a special case of the cyclotomic extensions over $\mathbb{Q}$.
			\end{rmk}
			\begin{rmk}\label{rmk10}
				It is also worth mentioning that, there are various applications of \textit{Dedekind Zeta Functions}, one of which is to obtain the well-known \textit{\textbf{Class Number Formula}}.
			\end{rmk}
			\begin{rmk}\label{rmk11}
				Another important property of these functions, which can often be said to be the relation between \textit{Dedekind Zeta Functions} and \textit{Dirichlet's L-functions} is that, Dedekind Zeta Functions can be factored into L-functions having a simpler functional equation.This is also an important fact linked to one of  the most famous conjectures \cite[p.~18]{2} in number theory:
				\begin{center}
					\textit{\textbf{ Artin's Conjecture on Analytic continuation of L-Series}}.
				\end{center} 
			\end{rmk}
			\vspace{160pt}
			\section*{Acknowledgments}
			I'll always be grateful to \textbf{Prof. A. Raghuram} ( Professor, Department of Mathematics, Fordham University, New York, USA ) for supervising my research project on this topic. His kind guidance helped me immensely in detailed understanding of this topic.

			\vspace{100pt}

	\end{document}